\documentclass[11pt,a4paper]{article}
\usepackage[utf8]{inputenc}
\usepackage[T1]{fontenc}
\usepackage[english]{babel}
\usepackage{textcomp}

\usepackage{amsmath,amssymb,amsopn,amsthm}
\usepackage{lmodern}
\usepackage[a4paper]{geometry}
\usepackage{graphicx}
\usepackage{xcolor}
\usepackage{microtype}
\usepackage{lipsum}
\usepackage{array}
\usepackage{mathtools}
\usepackage{tikz}
\usepackage{enumitem}

\usepackage{hyperref}

\newcommand{\F}{\mathbb{F}}

\theoremstyle{plain}
\newtheorem{theoreme}{Theorem}[section]
\newtheorem*{theoreme*}{Theorem}
\newtheorem{proposition}[theoreme]{Proposition}
\newtheorem*{proposition*}{Proposition}
\newtheorem{corollaire}[theoreme]{Corollary}
\newtheorem{lemme}[theoreme]{Lemma}

\theoremstyle{definition}

\newtheorem{remarque}[theoreme]{Remark}
\newtheorem{exemple}[theoreme]{Example}

\title{\textbf{Explicit formulas and exact values for the number of rational points on singular curves over finite fields}}

\author{\textsc{Lorenzo Beninati}\thanks{This work was partially supported by the French National Research Agency under grant ANR-21-CE39-0009-BARRACUDA.}}

\date{}

\begin{document}

\maketitle

\begin{abstract}
We provide new explicit formulas for bounding the number of rational points on singular curves over finite fields. This enables us to obtain exact values of $N_q(g,\pi)$ which is defined as the maximum number of rational points over $\F_q$ on a curve of geometric genus $g$ and arithmetic genus $\pi$. We also give special attention to the case $g=2$ in order to extend the work of Aubry and Iezzi on $N_q(0,\pi)$ and $N_q(1,\pi)$.
\end{abstract} 

\section{Introduction}

The study of the number of rational points on curves is a classical topic in algebraic geometry, with many applications in cryptography or coding theory for example. Throughout this paper, by the word \textit{curve} we shall mean an absolutely irreducible projective algebraic curve. According to Weil (see \cite{weil}), for any integer $n\geq 1$, the number of $\F_{q^n}$-rational points on a smooth curve $X$ defined over $\F_q$ of genus $g$ can be expressed as
\begin{equation} \label{weil}
\# X(\F_{q^n}) = q^n+1 - \sum_{j=1}^g({\omega_j}^n+{\overline{\omega}_j}^n),
\end{equation}
where $\omega_j$ are the reciprocal roots of the numerator polynomial $L_X(T)$ of the zeta function of $X$. Moreover, Weil proved that there are complex algebraic integers that satisfy the Riemann hypothesis for curves over finite fields, i.e. $\lvert \omega_j \rvert =\sqrt{q}$. From this expression, one can deduce a bound for the number of rational points, known as the Weil bound and given by
$$
\lvert \# X(\F_q) - (q +1) \rvert \leq 2g\sqrt{q}.
$$
However, this bound is not optimal in general, especially when $g$ is large with respect to $q$. Consequently, several methods have been studied to obtain sharper bounds. One of them relies on explicit formulas introduced by Serre in \cite{serre2}, whose main idea can be summarized as follows. Consider 
$$
f(\theta) = 1 + 2\sum_{n\geq 1}c_n \cos(n\theta)
$$ 
an even trigonometric polynomial to which we associate the following polynomials
$$
\psi(t) = \sum_{n\geq 1}c_n t^n \quad \text{and} \quad \psi_d(t) = \sum_{d \vert n}c_n t^n .
$$
We then have the following result (Theorem 5.3.2 in \cite{serre}).
\begin{theoreme}[Serre] \label{thm_serre}
Let $X$ be a smooth curve defined over $\F_q$ of genus $g$ and let $\omega_j$ be the reciprocal roots of $L_X(T)$ expressed in the form $\omega_j=\sqrt{q}e^{i\theta_j}$, with $0\leq \theta_j \leq \pi$. Then we have
$$
\sum_{j=1}^g f(\theta_j) + \sum_{d\geq 1} d B_d(X) \psi_d(q^{-1/2}) = g + \psi(q^{-1/2})+\psi(q^{1/2}) ,
$$
where $B_d(X)$ denotes the number of closed points of degree $d$ of $X$.
\end{theoreme}

Moreover, by adding extra hypotheses on $f$ it is possible to obtain an upper bound for the number of rational points. More precisely, we shall assume that $f(\theta)\geq 0$ for all $\theta \in \mathbb{R}$ and that $c_n \geq 0$ for all $n\geq 1$. Such a function is called \textit{doubly positive}. A direct consequence of Theorem \ref{thm_serre} is the following result.

\begin{corollaire}[Corollary 5.3.4 in \cite{serre}]\label{borne_exp_lisse}
Under the same assumptions as in Theorem \ref{thm_serre}, if we also suppose that $f$ is doubly positive then we have
$$
\# X(\F_q) \leq \frac{g}{\psi(q^{-1/2})}+1+\frac{\psi(q^{1/2})}{\psi(q^{-1/2})}.
$$
\end{corollaire}

The aim is then to find the best choice of the polynomial $f$ in order to obtain the best possible bound for the number of rational points of $X$. This optimization problem has been completely solved (for $q \neq 2$) by Oesterlé (see Chapter 6 in \cite{serre}). These optimized bounds, also called Oesterlé bounds, generally provide the sharpest estimates for $N_q(g)$ which is defined as the maximum number of rational points of a smooth curve defined over $\F_q$ of genus $g$.

One may then ask whether this method can be extended to potentially singular curves. An initial positive answer was given by Aubry and Perret (see \cite{aubry}), they state that
\begin{equation} \label{fe_aubry}
\# X(\F_q) \leq \frac{1}{\psi(q^{-1/2})}\bigg(g+\frac{\pi-g}{2}\bigg)+1+\frac{\psi(q^{1/2})}{\psi(q^{-1/2})},
\end{equation}
where $\pi$ denotes the arithmetic genus of the potentially singular curve $X$.

In this paper we show that it is possible to obtain another explicit formula for singular curves that always improve (\ref{fe_aubry}). More precisely, we prove the following proposition in the next section.

\begin{proposition*}[Proposition \ref{prop_fe_LT}]
Let $X$ be a curve defined over $\F_q$ of geometric genus $g$ and arithmetic genus $\pi$. We have
$$
\# X(\F_q) \leq \frac{1}{\psi(q^{-1/2})}\bigg(g+\frac{\pi-g}{\sqrt{q}+1}\bigg)+1+\frac{\psi(q^{1/2})}{\psi(q^{-1/2})}.
$$
\end{proposition*}

This new explicit formula enables us to obtain new upper bounds for $N_q(g,\pi)$. This quantity, defined by Aubry and Iezzi in \cite{iezzi} as a generalization of $N_q(g)$, denotes the maximum number of rational points on a curve defined over $\F_q$ of geometric genus $g$ and arithmetic genus $\pi$. Moreover, in the last section we show that, in certain cases, Proposition \ref{prop_fe_LT} provides not only sharper bounds but also new exact values for the quantity $N_q(g,\pi)$. These new values are summarized in the following result.

\begin{proposition*}[Proposition \ref{N_q(g,pi)}]
We have the following values for $N_q(g,\pi)$.\\
\noindent
\begin{minipage}[t]{0.33\textwidth}
\begin{itemize}[label=]
  \item $N_2(0,2)=4$,
  \item $N_2(5,6)=9$,
  \item $N_3(15,16)=28$,
  \item $N_4(26,27)=55$.
\end{itemize}
\end{minipage}
\begin{minipage}[t]{0.33\textwidth}
\begin{itemize}[label=]
  \item $N_2(1,2)=5$,
  \item $N_2(6,7)=10$,
  \item $N_3(15,17)=28$,
\end{itemize}
\end{minipage}
\begin{minipage}[t]{0.33\textwidth}
\begin{itemize}[label=]
  \item $N_2(4,5)=8$,
  \item $N_3(0,4)=7$,
  \item $N_4(9,10)=26$,
\end{itemize}
\end{minipage}
\end{proposition*}

Finally, motivated by the work of Aubry and Iezzi on $N_q(0,\pi)$ and $N_q(1,\pi)$ (Corollary 5.4 and Corollary 5.5 in \cite{iezzi}), we investigate the case of $N_q(2,\pi)$ with alternative methods. More precisely, we focus on values of $N_q(2,\pi)$ that attain the upper bound $N_q(2)+\pi-2$. This leads us to consider smooth curves having $N_q(2)$ rational points, originally studied by Serre (see Chapter 3 in \cite{serre}). In particular, we study what is called the \textit{defect} of such curves, which depends on the value of $q$.\\
The next theorem describes all values of $N_q(2,\pi)$ that reach the upper bound $N_q(2)+\pi-2$, according to the values of $q$ and $\pi$.

\begin{theoreme*}[Theorem \ref{g=2}]
Let $q$ be a power of a prime $p$. Denote by $m=[2\sqrt{q}]$ the integer part of $2\sqrt{q}$ and by $\{2\sqrt{q}\}$ its fractional part.\\
If $q$ is a square:
\begin{itemize}[label=\textbullet]
\item $N_4(2,\pi)=10+\pi-2$ if and only if $2 \leq \pi \leq 6$.
\item $N_9(2,\pi)=20+\pi-2$ if and only if $2 \leq \pi \leq 26$.
\item $N_q(2,\pi)=q+1+4\sqrt{q}+\pi-2$ if and only if $2 \leq \pi \leq 2+\frac{q^2-5q-4\sqrt{q}}{2} \text{ and } q\neq 4,9$.
\end{itemize}
If $q$ is not a square:
\begin{itemize}[label=\textbullet]
\item If $q$ is non-special, \\
$N_q(2,\pi)=q+1+2m+\pi-2$ if and only if $2 \leq \pi \leq 2+\frac{q(q+3)}{2}-m(m+1)$.
\item If $q$ is special and $\{2\sqrt{q}\}>\frac{\sqrt{5}-1}{2}$,\\
$N_q(2,\pi)=q+2m+\pi-2$ if and only if $2 \leq \pi \leq 2+\frac{q(q+3)}{2}-m^2-1$.
\item If $q$ is special and $\{2\sqrt{q}\}<\frac{\sqrt{5}-1}{2}$,
\begin{itemize}
\item $N_q(2,\pi)=q-1+2m+\pi-2$ if and only if $2 \leq \pi \leq 2+\frac{q(q+3)}{2}-m(m-1) \text{ and } q\neq 2^5,2^{13}$.
\item $N_q(2,\pi)=q-1+2m+\pi-2$ if and only if $2 \leq \pi \leq 2+\frac{q(q+3)}{2}-m(m-1)-1 \text{ and } q = 2^5,2^{13}$.
\end{itemize}

\end{itemize}
\end{theoreme*}

We conclude this paper with a brief but noteworthy example. We show that in some cases, studying the number of rational points on singular curves provides information about smooth curves, which highlights the relevance of their study.

\section{Explicit formulas for singular curves}

\subsection{Preliminaries}

We begin by recalling some generalities on singular curves. Let $X$ be a singular curve and let $\tilde{X}$ be its normalization. For any point $P \in X$, we denote by $\mathcal{O}_P$ the local ring of $X$ at $P$, by $\mathfrak{m}_P$ the maximal ideal of $\mathcal{O}_P$ and by $\text{deg}P = [\mathcal{O}_P/\mathfrak{m}_P : \F_q]$ the degree of $P$. If $\overline{\mathcal{O}_P}$ denotes the integral closure of $\mathcal{O}_P$ in the function field $\F_q(X)$ of $X$, we define the \textit{degree of singularity} $\delta_P$ of $P$ by
$$
\delta_P = \text{dim}_{\F_q}\overline{\mathcal{O}_P}/\mathcal{O}_P.
$$
In particular, $\delta_P=0$ if and only if $P$ is non-singular. Finally, since a curve has finitely many singular points, we can set
$$
\delta = \sum_{P \in X} \delta_P
$$
and define the \textit{arithmetic genus} $\pi$ by
$$
\pi = g+\delta,
$$
where $g$ is the \textit{geometric genus}, i.e. the genus of the normalization $\tilde{X}$. Aubry and Perret established in \cite{aubry} a relation between the number of rational points of a singular curve and that of its normalization, which reads
\begin{equation} \label{ine}
\vert \# \tilde{X}(\F_q) - \# X(\F_q) \vert \leq \pi-g.
\end{equation}
Moreover, they showed that the number of rational points of a singular curve is given by
\begin{equation} \label{pts_sing}
\# X(\F_{q^n}) = q^n +1 - \sum_{j=1}^g({\omega_j}^n+{\overline{\omega}_j}^n)- \sum_{k=1}^{\Delta_X}{\beta_k}^n,
\end{equation}
with $\omega_j$ the reciprocal roots of the numerator polynomial $L_{\tilde{X}}(T)$ of the zeta function of $\tilde{X}$, $\beta_k$ the reciprocal roots of the cyclotomic part of $L_X(T)$ and $\Delta_X = \#\tilde{X}(\overline{\F}_q) - \# X(\overline{\F}_q)$. Together with the inequality (\ref{ine}), this leads to the Aubry-Perret bound (Corollary 2.4 in \cite{aubry}), which generalizes the Weil bound and is given by
\begin{equation} \label{AP}
\vert \# X(\F_q) - (q +1) \vert \leq 2g\sqrt{q}+\pi-g.
\end{equation}

\subsection{Explicit formulas}

A first new explicit formula for the number of rational points on singular curves can be directly deduced from inequality (\ref{ine}) by applying Corollary \ref{borne_exp_lisse} to $\# \tilde{X}(\F_q)$. 

\begin{proposition} \label{prop_fes}
Let $X$ be a curve defined over $\F_q$ of geometric genus $g$ and arithmetic genus $\pi$. We have
\begin{equation} \label{fes}
\# X(\F_q) \leq \frac{g}{\psi(q^{-1/2})}+1+\frac{\psi(q^{1/2})}{\psi(q^{-1/2})}+\pi-g.
\end{equation}
\end{proposition}

\begin{remarque} \label{rmk_comp}
A comparison between the two singular explicit formulas shows that the formula proposed by Aubry and Perret (\ref{fe_aubry}) is sharper than that of Proposition \ref{prop_fes} if and only if $\psi(q^{-1/2})>1/2$.
However, for a suitable choice of polynomial $f$, that is, one for which our bounds are as tight as possible, this condition generally does not hold since the coefficients $c_n$ are always less than $1$ (Lemma 5.5.3 in \cite{serre}). Moreover, it is clear that this condition is even less likely to hold when $q$ is large. We detail the method to obtain such a suitable polynomial in Subsection \ref{oesterlé}.
\end{remarque}

We illustrate the previous remark with a simple example.

\begin{exemple} \label{ex_fe}
Consider the doubly positive polynomial $f(\theta)=1+\cos(\theta)$. Then $\psi(t)=\frac{t}{2}$ and hence $\psi(q^{-1/2}) = \frac{1}{2\sqrt{q}}< \frac{1}{2}$ and $\psi(q^{1/2})=\frac{\sqrt{q}}{2}$. Bound (\ref{fe_aubry}) gives
$$
\# X(\F_q) \leq q+1+2g\sqrt{q}+(\pi-g)\sqrt{q},
$$
whereas bound (\ref{fes}) reads
$$
\# X(\F_q) \leq q+1+2g\sqrt{q}+\pi-g.
$$
In particular we recover the Aubry-Perret bound (\ref{AP}) with this choice of polynomial, which is known to be optimal in some cases (see \cite{iezzi2}).
\end{exemple}

In the following proposition we provide another singular explicit formula which is always better than (\ref{fe_aubry}) and often better than (\ref{fes}).

\begin{proposition} \label{prop_fe_LT}
Let $X$ be a curve defined over $\F_q$ of geometric genus $g$ and arithmetic genus $\pi$. We have
\begin{equation} \label{fe_LT}
\# X(\F_q) \leq \frac{1}{\psi(q^{-1/2})}\bigg(g+\frac{\pi-g}{\sqrt{q}+1}\bigg)+1+\frac{\psi(q^{1/2})}{\psi(q^{-1/2})}.
\end{equation}
\end{proposition}

To prove Proposition \ref{prop_fe_LT} we need the following lemma, which is a result obtained by Lachaud and Tsfasman within their work on explicit formulas for the number of points on varieties defined over finite fields (see \cite{lachaud}).

\begin{lemme}[Lemma 3.8 in \cite{lachaud}] \label{LT}
For all $\theta \in \mathbb{R}$, we have the  inequality
$$
-\textnormal{Re}(\psi(q^{-1/2}e^{i\theta})) \leq \frac{1}{\sqrt{q}+1}.
$$
\end{lemme}

To make things clearer for the reader and because this lemma is not presented in this form in \cite{lachaud}, we provide a proof below. In fact, in \cite{lachaud} the result is stated in a more general framework, which is that of smooth algebraic varieties, whereas our proof is adapted to the context considered in this work.

\begin{proof}
We set $\tilde{f}(z)=1+2\text{Re}(\psi(z))$. In particular, we have $\tilde{f}(q^{-1/2}e^{i\theta}) = 1 + 2\displaystyle \sum_{n\geq 1}c_nq^{-n/2}\cos(n\theta)$. Since $\tilde{f}$ is the real part of a holomorphic function, it is harmonic. Moreover, the double positivity of $f$ implies that $\tilde{f}$ is positive on the unit disk. We can therefore apply the Herglotz's representation theorem and deduce that there exists a probability measure $\mu$ on the unit circle such that:
$$
\tilde{f}(q^{-1/2}e^{i\theta}) = \int_0^{2\pi} P_{q^{-1/2}}(\theta - \varphi)d\mu(\varphi),
$$
where $P_r$ is the Poisson kernel defined by $P_r(\theta)=  \frac{1-r^2}{1-2r\cos(\theta)+r^2}$. Consequently, we deduce that
$$
\text{Re}(\psi(q^{-1/2}e^{i \theta})) = \int_0^{2\pi} \frac{1}{2}(P_{q^{-1/2}}(\theta - \varphi)-1) d\mu(\varphi).
$$
Now, one has
$$
\frac{1-q^{-1/2}}{1+q^{-1/2}}=P_{q^{-1/2}}(\pi) \leq P_{q^{-1/2}}(\theta) \leq P_{q^{-1/2}}(0) = \frac{1+q^{-1/2}}{1-q^{-1/2}},
$$
hence
$$
\frac{-q^{-1/2}}{1+q^{-1/2}}\leq \frac{1}{2}(P_{q^{-1/2}}(\theta-\varphi)-1) \leq \frac{q^{-1/2}}{1-q^{-1/2}}.
$$
We deduce that $\text{Re}(\psi(q^{-1/2}e^{i \theta})) \geq \frac{-q^{-1/2}}{1+q^{-1/2}} = -\frac{1}{\sqrt{q}+1}$, i.e. $-\text{Re}(\psi(q^{-1/2}e^{i \theta})) \leq \frac{1}{\sqrt{q}+1}$.
\end{proof}

\begin{proof}[Proof of Proposition \ref{prop_fe_LT}]
We begin by recalling the following relation between rational points and closed points of the curve:
$$
\sum_{d \mid n}dB_d(X) = \# X(\F_{q^n}).
$$
This equality, together with (\ref{pts_sing}), gives
$$
\sum_{d \mid n}dB_d(X) = q^n +1 - 2q^{n/2}\sum_{j=1}^g\cos(n\theta_j) -\sum_{k=1}^{\Delta_X}{\beta_j}^n.
$$
Multiplying by $c_nq^{-n/2}$ and summing over $n$, we get
$$
\sum_{j=1}^g f(\theta_j) + \sum_{d\geq 1} d B_d(X) \psi_d(q^{-1/2}) = g + \psi(q^{-1/2})+\psi(q^{1/2})-\sum_{j=1}^{\Delta_X}\psi(q^{-1/2}\beta_j) .
$$
Furthermore, since the left-hand side is real we deduce that
$$
\sum_{j=1}^g f(\theta_j) + \sum_{d\geq 1} d B_d(X) \psi_d(q^{-1/2}) = g + \psi(q^{-1/2})+\psi(q^{1/2})-\sum_{j=1}^{\Delta_X}\text{Re}(\psi(q^{-1/2}\beta_j)) .
$$
Since $f$ is doubly positive, we obtain
$$
\# X(\F_q)\psi(q^{-1/2}) \leq g + \psi(q^{-1/2})+\psi(q^{1/2})-\sum_{j=1}^{\Delta_X}\text{Re}(\psi(q^{-1/2}\beta_j)) .
$$
The result follows by applying Lemma \ref{LT} and the inequality $\Delta_X \leq \pi -g$.
\end{proof}

\begin{remarque} \label{rmk_comp2}
The bound in Proposition \ref{prop_fe_LT} is always better than bound (\ref{fe_aubry}) since $\frac{1}{\sqrt{q}+1}< \frac{1}{2}$ for any prime power $q$.
\end{remarque}

\subsection{Doubly positive trigonometric polynomials} \label{oesterlé}

We describe here the optimization introduced by Oesterlé in order to obtain optimal trigonometric polynomials for Serre's explicit formula (see Corollary \ref{borne_exp_lisse}). Since this work has not been published by Oesterlé, the best reference about this topic can be found in Chapter 6 in \cite{serre}. Oesterlé's optimization is crucial for our work, as we will use some of these polynomials in the next section to study $N_q(g,\pi)$.

\hspace{3cm}

Let $\lambda$ be an integer greater than $q$. Define $r$ as the unique integer satisfying $\sqrt{q}^r<\lambda \leq \sqrt{q}^{r+1}$ and set 
$$
u =  \frac{\sqrt{q}^{r+1}-\lambda}{\lambda\sqrt{q}-\sqrt{q}^r}.
$$
Then there exists a unique real $\varphi_0$ such that
$$
\varphi_0 \in \bigg[\frac{\pi}{r+1},\frac{\pi}{r}\bigg[ \quad \text{ and } \quad \cos\bigg(\frac{r+1}{2}\varphi_0\bigg)+u\cos\bigg(\frac{r-1}{2}\varphi_0\bigg)=0.
$$

We recall the following positivity result.

\begin{proposition}[Proposition 6.2.3 in \cite{serre}]
For $1\leq n \leq r-1$, we set
$$
c_n = \frac{(r-n)\cos(n\varphi_0)\sin(\varphi_0)+\sin((r-n)\varphi_0)}{r\sin(\varphi_0)+\sin(r\varphi_0)}.
$$
Then, the trigonometric polynomial $f$ defined by
$$
f(\theta) = 1 + 2\sum_{n=1}^{m-1}c_n\cos(n\theta),
$$
is doubly positive.
\end{proposition}

\begin{exemple}
Let $r\geq 2$ and suppose that $\lambda = {\sqrt{q}}^{r+1}$. In this case the computations are straightforward. Indeed, $u=0$ and $\varphi_0 = \frac{\pi}{r+1}$, so the corresponding trigonometric polynomial depends only on the choice of $r$. For instance,
\begin{itemize}
\item $r=2$: $f(\theta)=1+\cos(\theta)$,

\item  $r=3$: $f(\theta)=1+\sqrt{2}\cos(\theta)+\frac{1}{2}\cos(2\theta)$,

\item $r=4$: $f(\theta)=1+\frac{1+\sqrt{5}}{2}\cos(\theta)+\frac{2\sqrt{5}}{5}\cos(2\theta)+\frac{5-\sqrt{5}}{10}\cos(3\theta)$ and

\item $r=5$: $f(\theta)=1+\sqrt{3}\cos(\theta)+\frac{7}{6}\cos(2\theta)+\frac{\sqrt{3}}{3}\cos(3\theta)+\frac{1}{6}\cos(4\theta)$.
\end{itemize}

\end{exemple}

In the rest of the paper, we will use the doubly positive polynomials introduced in the previous example. For each polynomial, it is therefore necessary to determine which formula among (\ref{fe_aubry}), Proposition \ref{prop_fes} and Proposition \ref{prop_fe_LT} provides the best bound.\\
As already mentioned in Remark \ref{rmk_comp2}, the explicit formula given in Proposition \ref{prop_fe_LT} is always better than (\ref{fe_aubry}). Hence, it suffices to compare explicit formulas from Propositions \ref{prop_fes} and \ref{prop_fe_LT}. As in Remark \ref{rmk_comp}, we can show that Proposition \ref{prop_fe_LT} gives better results than Proposition \ref{prop_fes} if and only if $\psi(q^{-1/2})>\frac{1}{\sqrt{q}+1}$. For $r=2$, it is easy to check that this condition is never satisfied, regardless of the value of $q$. However, in this case the bound of Proposition \ref{prop_fes} coincides with the Aubry-Perret bound (see Example \ref{ex_fe}), which is of no interest here since our goal is precisely to obtain new estimates. Moreover, after some tedious computations, one can show that for $r=3,4$ and $5$, the condition holds respectively for  $q<13$, $q\leq 49$ and $q<131$. In practice, we will be interested in relatively small values of $q$.  Consequently, Proposition \ref{prop_fe_LT} will always provide the best estimate in the situations of interest, and we shall therefore restrict ourselves to this explicit formula in the rest of the paper.

\section{Some exact values of $N_q(g,\pi)$}

\subsection{Exacts values deduced from explicit formulas}

We begin by recalling some informations about $N_q(g,\pi)$. It is defined by Aubry and Iezzi (see \cite{iezzi}) as the maximum number of rational points on a curve defined over $\F_q$ of geometric genus $g$ and arithmetic genus $\pi$. In particular, it is a generalization of $N_q(g)$ since $N_q(g)=N_q(g,g)$. Moreover, the Aubry-Perret bound (\ref{AP}) gives
$$
N_q(g,\pi) \leq q+1+2g\sqrt{q}+\pi-g.
$$

In this section, we determine exact values of $N_q(g,\pi)$ for arbitrary triples $(q,g,\pi)$ using Proposition \ref{prop_fe_LT} and the doubly positive polynomials presented in the previous section. Since $N_q(g) \leq N_q(g,\pi)$ (Proposition 4.1 in \cite{iezzi}), sufficiently precise upper bounds obtained via explicit formulas allow us to determine an exact value of $N_q(g,\pi)$, provided that $N_q(g)$ is known. This is the purpose of the next proposition.

\begin{proposition} \label{N_q(g,pi)}
We have the following values for $N_q(g,\pi)$.\\
\noindent
\begin{minipage}[t]{0.33\textwidth}
\begin{itemize}[label=]
  \item $N_2(0,2)=4$,
  \item $N_2(5,6)=9$,
  \item $N_3(15,16)=28$,
  \item $N_4(26,27)=55$.
\end{itemize}
\end{minipage}
\begin{minipage}[t]{0.33\textwidth}
\begin{itemize}[label=]
  \item $N_2(1,2)=5$,
  \item $N_2(6,7)=10$,
  \item $N_3(15,17)=28$,
\end{itemize}
\end{minipage}
\begin{minipage}[t]{0.33\textwidth}
\begin{itemize}[label=]
  \item $N_2(4,5)=8$,
  \item $N_3(0,4)=7$,
  \item $N_4(9,10)=26$,
\end{itemize}
\end{minipage}
\end{proposition}

\begin{proof}
Let us detail the case $q=2$ and $g=4$. It is known that $N_2(4)=8$, hence $N_2(4,\pi) \geq 8$ for every $\pi \geq 4$. Consider now the trigonometric polynomial 
$$
f(\theta)=1+2c_1\cos(\theta)+2c_2\cos(2\theta)+2c_3\cos(3\theta)+2c_4\cos(4\theta)
$$ 
with $c_1=\frac{\sqrt{3}}{2}$, $c_2=\frac{7}{12}$, $c_3=\frac{\sqrt{3}}{6}$ and $c_4=\frac{1}{12}$. With this choice Proposition \ref{prop_fe_LT} gives $N_2(4,5) \leq 8$ and thus $N_2(4,5) = 8$.\\
We proceed in the same way to determine the other values of $N_q(g,\pi)$. The values $N_2(0,2)$, $N_2(1,2)$ and $N_3(0,4)$ are found by applying Proposition \ref{prop_fe_LT} with the polynomial $f(\theta) = 1 + \sqrt{2}\cos(\theta) + \frac{1}{2}\cos(2\theta)$. Moreover, for $N_2(0,2)$ and $N_3(0,4)$ we also use the values of $N_2(0,1)$ and $N_3(0,3)$ given in Corollary 5.4 in \cite{iezzi}. The value of $N_4(9,10)$ is obtained with the polynomial $f(\theta) = 1 + \frac{1+\sqrt{5}}{2}\cos(\theta) + \frac{2\sqrt{5}}{5}\cos(2\theta)+ \frac{5-\sqrt{5}}{10}\cos(3\theta)$. Finally, the remaining values are deduced by using the same polynomial as for $N_2(4,5)$.
\end{proof}

\begin{remarque}
We have relied on data available on the website \url{www.manypoints.org} (see \cite{manypoints}) for the values of $N_q(g)$ quoted in the proof.
\end{remarque}

\subsection{The case $g=2$}

There are other methods to determine values of $N_q(g,\pi)$, especially when one is interested in so-called $\delta$\textit{-optimal} curves (see \cite{iezzi}). Recall that $\delta = \pi -g$ and that a curve $X$ is said to be $\delta$-optimal if $\#X(\F_q)=N_q(g)+\pi-g$. Moreover, if the curve reaches the Aubry-Perret bound (\ref{AP}), i.e. if $\#X(\F_q)=q+1+2g\sqrt{q}+\pi-g$, then $X$ is called \textit{maximal}. The following result deals with the number of rational points on such curves.

\begin{theoreme}[Theorem 5.3 in \cite{iezzi}] \label{thm_anna}
Let $X$ be a curve defined over $\F_q$ of geometric genus $g$ and arithmetic genus $\pi$. Then $X$ is $\delta$-optimal if and only if its normalization satisfies 
$$
\#\tilde{X}(\F_q)=N_q(g),
$$
and
$$
\pi-g\leq B_2(\mathcal{X}_q(g)),
$$
where $B_2(\mathcal{X}_q(g))$ denotes the maximal number of points of degree $2$ on the set $\mathcal{X}_q(g)$ of smooth optimal curves defined over $\F_q$ of genus $g$.
\end{theoreme}

Theorem \ref{thm_anna} tells us that if one wants to know the exact values $N_q(g,\pi)$ in the case where $N_q(g,\pi)=N_q(g)+\pi-g$, it is essential to determine $B_2(\mathcal{X}_q(g))$. In particular, the theorem specializes when one considers maximal curves. Indeed, in this situation the number of points of degree $2$ on Weil-maximal curves is known explicitly (Proposition 4.1 in \cite{iezzi2}).\\
In this section we focus more particularly on $N_q(2,\pi)$. In the smooth case, the values of $N_q(2)$ have been completely determined by Serre (see Chapter 3 in \cite{serre}). The goal here is therefore to obtain information about the maximal number of points of degree $2$ of smooth curves that have $N_q(2)$ rational points. One way to do this is to study the defect of these curves. Recall that the \textit{defect} of a smooth curve $X$ is defined as the difference $q+1+gm-\#X(\F_q)$, where $m=[2\sqrt{q}]$. For curves with $N_q(2)$ rational points, the defect cannot exceed $3$ (Theorem 3.2.3 in \cite{serre}), hence the number of possible values of $x_1$ and $x_2$, where $x_i=\omega_i+\overline{\omega}_i$ for $i=1,2$, is limited. By finding these $x_i$ we can then determine $B_2(\mathcal{X}_q(2))$.\\
The possible values of $(x_1,x_2)$ were fully studied by Smyth (see \cite{smyth}) up to defect $6$. For our purpose, namely the case of genus $2$ curves of defect at most $3$, the different possibilities are summarized in the following lemma.

\begin{lemme}[Smyth, \cite{smyth}] \label{defect}
Let $q$ be a power of a prime and let $m=[2\sqrt{q}]$. Then the values of $(x_1,x_2)$ are as follow:\\
\begin{minipage}[t]{0.5\textwidth}
\textbullet Defect $0$: \begin{itemize}[label=]
\item $\pm(m,m)$.
\end{itemize}
\end{minipage}
\begin{minipage}[t]{0.5\textwidth}
\textbullet Defect 1: \begin{itemize}[label=]
\item $\pm(m,m-1)$,
\item $\pm\bigg(m+\frac{\sqrt{5}-1}{2},m+\frac{-1-\sqrt{5}}{2}\bigg)$.
\end{itemize}
\end{minipage}
\begin{minipage}[t]{0.5\textwidth}
\textbullet Defect $2$:\begin{itemize}[label=]
\item $\pm(m,m-2)$,
\item $\pm(m-1,m-1)$,
\item $\pm(m-1+\sqrt{2},m-1-\sqrt{2})$,
\item $\pm(m-1+\sqrt{3},m-1-\sqrt{3})$.
\end{itemize}
\end{minipage}
\begin{minipage}[t]{0.5\textwidth}
\textbullet Defect $3$:\begin{itemize}[label=]
\item $\pm(m,m-3)$,
\item $\pm(m-1,m-2)$,
\item $\pm\bigg(m+\frac{\sqrt{21}-3}{2},m+\frac{-3-\sqrt{21}}{2}\bigg)$,
\item $\pm\bigg(m+\frac{\sqrt{17}-3}{2},m+\frac{-3-\sqrt{17}}{2}\bigg)$,
\item $\pm\bigg(m+\frac{\sqrt{13}-3}{2},m+\frac{-3-\sqrt{13}}{2}\bigg)$,
\item $\pm\bigg(m+\frac{\sqrt{5}-3}{2},m+\frac{-3-\sqrt{5}}{2}\bigg)$.
\end{itemize}
\end{minipage}

\end{lemme}

\begin{remarque}
Lemma \ref{defect} is not presented in the same way as in \cite{smyth}. In his paper, Smyth studies totally positive algebraic integers and, for each defect, provides a list of polynomials whose roots $\lambda_i$ are these alebraic integers. In our context, it should be noted that each $x_i$, which is defined as twice the integer part of the algebraic integer $\omega_i$, can be expressed in the form $x_i = m+1-\lambda_i$. For further details, see Chapter 2 in \cite{serre}.
\end{remarque}

Of course, not all these values necessarily correspond to a curve. For instance, one must also have $\vert x_i \vert \leq 2\sqrt{q}$. Moreover, a result of Serre about the indecomposability of Jacobians further restricts the possible cases.

\begin{theoreme}[Theorem 2.4.1 in \cite{serre}] \label{jac}
Let $A$ be an abelian variety of dimension $g$ defined over $\F_q$. Suppose that $\{1,\dots,g\}$ can be partitioned into two non-empty subsets $I$ and $J$ such that
\begin{enumerate}
\item the algebraic integers $(x_i)_{i \in I}$ are permuted by $\textnormal{Gal}(\overline{\mathbb{Q}}/\mathbb{Q})$ (multiplicities included), and the same holds for the $(x_j)_{j \in J}$.
\item for every $i \in I$ and every $j \in J$, the algebraic integer $x_i-x_j$ is an algebraic unit.
\end{enumerate}
Then $A$ is not a Jacobian, i.e. the $x_i$ do not arise from a curve.
\end{theoreme}

We can now state the following theorem, which provides all exact values of $N_q(2,\pi)$ that equal $N_q(2)+\pi-2$. We consider different cases according to $q$, in particular the case where $q$ is \textit{special}. Recall that $q$ is said to be special if $q=p^e$ is not a square and if it satisfies one of the following two conditions:
\begin{enumerate}[label= \roman*)]
\item $p$ divides $m=[2\sqrt{q}]$.
\item $q$ can be represented by one of the polynomials $x^2+1$, $x^2+x+1$, $x^2+x+2$.
\end{enumerate}

\begin{theoreme} \label{g=2}
Let $q$ be a power of a prime $p$. Denote by $m=[2\sqrt{q}]$ the integer part of $2\sqrt{q}$ and by $\{2\sqrt{q}\}$ its fractional part.\\
If $q$ is a square:
\begin{itemize}[label=\textbullet]
\item $N_4(2,\pi)=10+\pi-2$ if and only if $2 \leq \pi \leq 6$.
\item $N_9(2,\pi)=20+\pi-2$ if and only if $2 \leq \pi \leq 26$.
\item $N_q(2,\pi)=q+1+4\sqrt{q}+\pi-2$ if and only if $2 \leq \pi \leq 2+\frac{q^2-5q-4\sqrt{q}}{2} \text{ and } q\neq 4,9$.
\end{itemize}
If $q$ is not a square:
\begin{itemize}[label=\textbullet]
\item If $q$ is non-special, \\
$N_q(2,\pi)=q+1+2m+\pi-2$ if and only if $2 \leq \pi \leq 2+\frac{q(q+3)}{2}-m(m+1)$.
\item If $q$ is special and $\{2\sqrt{q}\}>\frac{\sqrt{5}-1}{2}$,\\
$N_q(2,\pi)=q+2m+\pi-2$ if and only if $2 \leq \pi \leq 2+\frac{q(q+3)}{2}-m^2-1$.
\item If $q$ is special and $\{2\sqrt{q}\}<\frac{\sqrt{5}-1}{2}$,
\begin{itemize}
\item $N_q(2,\pi)=q-1+2m+\pi-2$ if and only if $2 \leq \pi \leq 2+\frac{q(q+3)}{2}-m(m-1) \text{ and } q\neq 2^5,2^{13}$.
\item $N_q(2,\pi)=q-1+2m+\pi-2$ if and only if $2 \leq \pi \leq 2+\frac{q(q+3)}{2}-m(m-1)-1 \text{ and } q = 2^5,2^{13}$.
\end{itemize}

\end{itemize}
\end{theoreme}

\begin{remarque}
Note that the case where $q$ is special and $\{2\sqrt{q}\} = \frac{\sqrt{5}-1}{2}$ is not covered by the theorem since Serre noticed that this situation cannot occur (Theorem 3.2.3 in \cite{serre}).
\end{remarque}

\begin{proof}
We proceed by considering each case of the theorem separately.\\
If $q$ is a square :
\begin{itemize}[label=\textbullet]
\item $q=4$. Let $X$ be a smooth curve that has $N_4(2) =10$ rational points. This curve has defect $3$, so according to Lemma \ref{defect} there are six possibilities for the values of $(x_1,x_2)$. However, using the fact that $\vert x_i \vert \leq 2\sqrt{q}=4$, one can restrict to the three cases
$$
-(m,m-3), \quad -(m-1,m-2) \quad \text{ and } \quad -\bigg(m+\frac{-3+\sqrt{5}}{2},m+\frac{-3-\sqrt{5}}{2}\bigg).
$$
Moreover, by Theorem \ref{jac}, the case $-(m-1,m-2)$ does not correspond to a curve. Indeed, $m-1-(m-2)=1$ is an algebraic unit. Thus only $-(m,m-3)$ and $-\left(m+\frac{-3+\sqrt{5}}{2},m+\frac{-3-\sqrt{5}}{2}\right)$ remain. In the first case, such a curve has $3$ points of degree $2$ whereas in the second case it has $4$. The curve given by the equation $y^2+y=\frac{x}{x^3+x+1}$ is a genus $2$ curve which has $10$ rational points over $\F_4$ (see Section 3.4 in \cite{serre}) and $4$ points of degree $2$. We conclude by applying Theorem \ref{thm_anna}.
\item $q=9$. Let $X$ be a smooth curve with $N_9(2) =20$ rational points. This curve has defect $2$, so there are four possibilities for the values of $(x_1,x_2)$. As in the previous case, noting that $\vert x_i \vert \leq 2\sqrt{q}=6$ leads to the two cases
$$
-(m,m-2) \quad \text{ and } \quad -(m-1,m-1).
$$
A curve such that $(x_1,x_2)=-(m,m-2)$ has $23$ points of degree $2$ whereas a curve such that $(x_1,x_2)=-(m-1,m-1)$ has $24$ points of degree $2$. Hence, it suffices to find a curve having $24$ points of degree $2$, which is the case of the curve $y^2=x^6+x^4+x^2-1$ (see Section 3.4 in \cite{serre}).
\item $q \neq 4,9$. According to Serre, in this case $N_q(2)=q+1+4\sqrt{q}$. The corresponding curves are Serre-Weil maximal curves, so $(x_1,x_2)$ is explicitly known to be $-(m,m)$. Such curves have $\frac{q^2-5q-4\sqrt{q}}{2}$ points  of degree $2$, which yields the result.
\end{itemize}
If $q$ is not a square:
\begin{itemize}[label=\textbullet]
\item $q$ non-special. In this case one knows that $N_q(2)=q+1+2m$. Such a curve has defect $0$, and we conclude as in the previous case.
\item $q$ special and $\{2\sqrt{q}\}>\frac{\sqrt{5}-1}{2}$. Then $N_q(2)=q+2m$, which corresponds to a defect $1$ case. The pair $(x_1,x_2)$ of a genus $2$ curve which has defect $1$ can only take the value $-(m+\frac{\sqrt{5}-1}{2},m+\frac{-1-\sqrt{5}}{2})$. Indeed, the case $-(m,m-1)$ is impossible by Theorem \ref{jac}. Thus such curves have $\frac{q(q+3)}{2}-m^2-1$ points of degree $2$.
\item $q$ special and $\{2\sqrt{q}\}<\frac{\sqrt{5}-1}{2}$. Assume $q\geq 5$. Serre showed that in the case where $p \mid m$, then $(x_1,x_2)$ can take the value $-(m-1,m-1)$. Likewise, if $p \nmid m$ and $p\neq 2$, then $-(m-1,m-1)$ is still possible. Finally, in the case where $p \nmid m$ and $p = 2$, which corresponds to $q=2^5$ or $2^{13}$, then $-(m,m-2)$ is possible. We conclude by computing the respective numbers of points  of degree $2$.\\
It remains to deal with the case $q=3$. Consider the curve $y^2=x^6+x^4+x^2+1$. This curve has $8$ rational points, hence it has defect $2$ (Theorem 3.7.1 in \cite{serre}). As shown by Serre, the couple $(x_1,x_2)$ can take the values
$$
-(m,m-2), \quad -(m-1,m-1) \quad \text{ and } \quad -(m+\sqrt{2}-1,m-\sqrt{2}-1).
$$
Moreover, one can show that this curve has $3$ points  of degree $2$, which corresponds to the case $-(m-1,m-1)$. We can thus deduce the result since the case $-(m-1,m-1)$ is the one that yields the largest possible number of points of degree $2$.
\end{itemize}
\end{proof}

\subsection{Results on smooth curves deduced from singular curves}

In this part we consider an example to show that explicit formulas can also indirectly provide information on the values of the algebraic integers $x_i=\omega_i+\overline{\omega}_i$ of the normalization of a curve.\\ Set $q=2$ and $g=3$. By applying Proposition \ref{prop_fe_LT} with the doubly positive polynomial $f(\theta)=1+\sqrt{3}\cos(\theta)+\frac{7}{6}\cos(2\theta)+\frac{\sqrt{3}}{3}\cos(3\theta)+\frac{1}{6}\cos(4\theta)$, we obtain the following bounds
$$
N_2(3,4) \leq 7, \quad N_2(3,5) \leq 8,\quad N_2(3,6) \leq 8 \quad \text{ and } \quad N_2(3,7) \leq 8.
$$ 
Moreover, since $N_2(3)=7$, it follows that $N_2(3,4) = 7$. Consequently, if $X$ is a smooth curve of genus $3$ defined over $\F_2$ with exactly $7$ rational points, then necessarily $B_2(X)=0$. Otherwise, if $B_2(X)\geq 1$, then by Theorem \ref{thm_anna} one could construct a singular curve of  geometric genus $g=3$ and arithmetic genus $\pi=4$ with $8$ rational points. This would contradict the fact that $N_2(3,4) = 7$. By a similar argument, the inequality $N_2(3,7) \leq 8$ implies $B_3(X) \leq 1$. This comes from a construction of singular curves proposed by Aubry and Iezzi (Theorem 4.2 in \cite{iezzi}). More particularly, if we suppose that $B_3(X) \geq 2$, then one could construct a singular curve having $9$ rational points and such that $g=3$ and $\pi=7$. This would contradict the inequality $N_2(3,7) \leq 8$. We then have $\#X(\F_2)=\#X(\F_4)=7$ and $\#X(\F_8)=7$ or $10$. Suppose first that $\#X(\F_8)=7$, the expression (\ref{weil}) gives us the following system
$$
\begin{cases}
\begin{array}{llllllr}
x_1 &+ &x_2 &+ &x_3 & = & -4 \\
{x_1}^2 &+ &{x_2}^2 &+ &{x_3}^2 & = & 10 \\
{x_1}^3 &+ &{x_2}^3 &+ &{x_3}^3 & = & -22
\end{array}
\end{cases}
$$
whose solution, up to permutation, is
$$
x_1=-2, \quad x_2=\sqrt{2}-1 \quad \text{ and } \quad x_3=-\sqrt{2}-1.
$$
Back to the previous notation, this gives
$$
(x_1,x_2,x_3) = -(m,m-\sqrt{2}-1,m+\sqrt{2}-1).
$$
However, by Theorem \ref{jac} we know that this triple of algebraic integers does not correspond to a curve. We deduce that $\#X(\F_8) \neq 7$ and hence $\#X(\F_8)=10$. This proves the following proposition.

\begin{proposition}
We have the following values for $N_q(g,\pi)$.\\
\noindent
\begin{minipage}[t]{0.33\textwidth}
\begin{itemize}[label=]
  \item $N_2(3,5)=8$,
\end{itemize}
\end{minipage}
\begin{minipage}[t]{0.33\textwidth}
\begin{itemize}[label=]
  \item $N_2(3,6)=8$,
\end{itemize}
\end{minipage}
\begin{minipage}[t]{0.33\textwidth}
\begin{itemize}[label=]
  \item $N_2(3,7)=8$.
\end{itemize}
\end{minipage}
\end{proposition}

Finally, by solving the system once again, now with $\#X(\F_8)=10$, we deduce the following corollary.

\begin{corollaire}
A smooth curve defined over $\F_2$ of genus $3$ is optimal, i.e. has $7$ rational points, if and only if
$$
(x_1,x_2,x_3) = -\bigg(3-4\cos^2\bigg(\frac{\pi}{7}\bigg),3-4\cos^2\bigg(\frac{2\pi}{7}\bigg),3-4\cos^2\bigg(\frac{3\pi}{7}\bigg)\bigg).
$$
\end{corollaire}

This result enables us to compute the zeta function of such a curve. Recall that in the case of a smooth curve $X$ defined over $\F_q$ of genus $g$, the zeta function of $X$ is a rational function of the form
$$
Z_X(T) =  \frac{ \prod_{j=1}^g(1-\omega_jT)(1-\overline{\omega}_jT)}{(1-T)(1-qT)}= \frac{ \prod_{j=1}^g(1-x_jT+qT^2)}{(1-T)(1-qT)}.
$$
In our example, $X$ is a smooth curve of genus $3$ defined over $\F_2$. Consequently,
$$
Z_X(T) = \frac{1+4T+9T^2+15T^3+18T^4+16T^5+8T^6}{(1-T)(1-2T)}.
$$
Note that this result is not new in the literature. Indeed, it is known that a genus $3$ curve is either hyperelliptic or a plane quartic. Hyperelliptic curves over $\F_2$ have at most $6$ rational points and plane curves cannot have more rational points than the projective plane itself, namely $7$. Moreover, Dickson (see \cite{dickson}) proved that there exists a unique plane quartic containing all $7$ rational points, this is the curve given by the equation
$$
x^3y+x^2y^2+xz^3+x^2z^2+y^3z+yz^3=0.
$$
One could easily recover the zeta function starting from this curve.

This example is not the only one for which explicit formulas on singular curves make it possible to obtain information about smooth curves. For instance, in the case where $q = 2$ and $g = 4$, this method shows that the algebraic integers $x_1,\dots, x_4$ of a smooth curve defined over $\F_2$ of genus $4$ having $N_2(4) = 8$ rational points must necessarily be of a permutation of 
$$
-\left(\frac{3+\sqrt{5}}{2},\frac{3-\sqrt{5}}{2},1-\sqrt{2},1+\sqrt{2}\right) \quad \text{or} \quad -(1-\sqrt{3},1+\sqrt{3},1,2).
$$
This is consistent with recent results obtained by Xarlès who computed equations of all genus $4$ curves defined over $\F_2$ up-to-isomorphism in \cite{xarles} and showed that there is actually only one curve with $8$ rational points. This curve corresponds to the 4-tuple
$$
(x_1,x_2,x_3,x_4) = -(1-\sqrt{3},1+\sqrt{3},1,2).
$$
In the same manner, when $q=2$ and $g=5$, the method provides three possibilities for the values of $x_1,\dots,x_5$. Again, following the work of Xarlès, Dragutinovi\'c showed that in fact only one of these possibilities corresponds to a curve since there is a unique genus $5$ curve defined over $\F_2$ with $N_2(5)=9$ rational points (see \cite{dragutinovic}). The corresponding $x_i$ are
$$
(x_1,x_2,x_3,x_4,x_5) = -(0,1-\sqrt{3},1+\sqrt{3},2,2).
$$
Nevertheless, these examples show that in certain cases explicit formulas can be very precise. By combining these estimates with the approach of Aubry and Iezzi, one can not only deduce new values of $N_q(g,\pi)$ but also obtain information about smooth curves. This makes the study of the number of rational points on singular curves even more interesting.\\ \

\noindent \textbf{Acknowledgement.} The author would like to express his sincere gratitude to his PhD supervisors, Yves Aubry and Fabien Herbaut. Their guidance, support and attentive listening were essential to the completion of this work. The author also wishes to thank Emmanuel Hallouin and Philippe Moustrou for valuable discussions on Oesterlé's optimization, as well as Marc Perret for useful comments on the paper.\\

\noindent \textsc{institut de mathématiques de toulon - imath, université de toulon, france}\\
\textit{Email adress:} \texttt{lorenzo-beninati@etud.univ-tln.fr}

\end{document}